\documentclass[11pt]{amsart} \textwidth=14.5cm \oddsidemargin=1cm
\usepackage{amsmath,wasysym}
\usepackage{amsxtra}
\usepackage{amscd}
\usepackage{amsthm}
\usepackage{amsfonts}
\usepackage{amssymb}
\usepackage{eucal}
\usepackage{graphics, color}
\usepackage{hyperref,mathrsfs}

\input prepictex
\input pictex
\input postpictex

\newtheorem{thm}{Theorem}[section]
\newtheorem{lem}[thm]{Lemma}
\newtheorem{cor}[thm]{Corollary}
\newtheorem{prop}[thm]{Proposition}

\theoremstyle{definition}

\newtheorem{conjecture}[thm]{Conjecture}

\theoremstyle{remark}
\newtheorem{rem}[thm]{Remark}

\numberwithin{equation}{section}

\begin{document}

\newcommand{\thmref}[1]{Theorem~\ref{#1}}
\newcommand{\secref}[1]{Section~\ref{#1}}
\newcommand{\lemref}[1]{Lemma~\ref{#1}}
\newcommand{\propref}[1]{Proposition~\ref{#1}}
\newcommand{\corref}[1]{Corollary~\ref{#1}}
\newcommand{\remref}[1]{Remark~\ref{#1}}
\newcommand{\eqnref}[1]{(\ref{#1})}
\newcommand{\exref}[1]{Example~\ref{#1}}

\newcommand{\nc}{\newcommand}
\nc{\Z}{{\mathbb Z}}
\nc{\hZ}{{\underline{\mathbb Z}}}
\nc{\C}{{\mathbb C}}
\nc{\N}{{\mathbb N}}
\nc{\F}{{\mf F}}
\nc{\Q}{\mathbb {Q}}
\nc{\la}{\lambda}
\nc{\ep}{\delta}
\nc{\h}{\mathfrak h}
\nc{\n}{\mf n}
\nc{\A}{{\mf a}}
\nc{\G}{{\mathfrak g}}
\nc{\SG}{\overline{\mathfrak g}}
\nc{\DG}{\widetilde{\mathfrak g}}
\nc{\D}{\mc D}
\nc{\Li}{{\mc L}}
\nc{\La}{\Lambda}
\nc{\is}{{\mathbf i}}
\nc{\V}{\mf V}
\nc{\bi}{\bibitem}
\nc{\NS}{\mf N}
\nc{\dt}{\mathord{\hbox{${\frac{d}{d t}}$}}}
\nc{\E}{\EE}
\nc{\ba}{\tilde{\pa}}
\nc{\half}{\frac{1}{2}}
\nc{\mc}{\mathcal}
\nc{\mf}{\mathfrak} \nc{\hf}{\frac{1}{2}}
\nc{\hgl}{\widehat{\mathfrak{gl}}}
\nc{\gl}{{\mathfrak{gl}}}
\nc{\hz}{\hf+\Z}
\nc{\dinfty}{{\infty\vert\infty}}
\nc{\SLa}{\overline{\Lambda}}
\nc{\SF}{\overline{\mathfrak F}}
\nc{\SP}{\overline{\mathcal P}}
\nc{\U}{\mathfrak u}
\nc{\SU}{\overline{\mathfrak u}}
\nc{\ov}{\overline}
\nc{\wt}{\widetilde}
\nc{\wh}{\widehat}
\nc{\sL}{\ov{\mf{l}}}
\nc{\sP}{\ov{\mf{p}}}
\nc{\osp}{\mf{osp}}
\nc{\spo}{\mf{spo}}
\nc{\hosp}{\widehat{\mf{osp}}}
\nc{\hspo}{\widehat{\mf{spo}}}
\nc{\hh}{\widehat{\mf{h}}}
\nc{\even}{{\bar 0}}
\nc{\odd}{{\bar 1}}
\nc{\mscr}{\mathscr}

\newcommand{\blue}[1]{{\color{blue}#1}}
\newcommand{\red}[1]{{\color{red}#1}}
\newcommand{\green}[1]{{\color{green}#1}}
\newcommand{\white}[1]{{\color{white}#1}}


\newcommand{\aaf}{\mathfrak a}
\newcommand{\bb}{\mathfrak b}
\newcommand{\cc}{\mathfrak c}
\newcommand{\qq}{\mathfrak q}
\newcommand{\qn}{\mathfrak q (n)}
\newcommand{\UU}{\bold U}
\newcommand{\VV}{\mathbb V}
\newcommand{\WW}{\mathbb W}
\nc{\TT}{\mathbb T}
\nc{\EE}{\mathbb E}
\nc{\FF}{\mathbb F}
\nc{\KK}{\mathbb K}
\nc{\rl}{\texttt{r,l}}

 \advance\headheight by 2pt

\title{Supercharacters of queer Lie superalgebras}

\author[Cheng]{Shun-Jen Cheng$^\dagger$}
\thanks{$^\dagger$Partially supported by a MoST and an Academia Sinica Investigator grant}
\address{Institute of Mathematics, Academia Sinica, Taipei,
Taiwan 10617} \email{chengsj@gate.sinica.edu.tw}

\begin{abstract}
Let $\G=\G_\even\oplus\G_\odd$ be the queer Lie superalgebra and let $L$ be a finite-dimensional non-trivial irreducible $\G$-module. Restricting the $\G$-action on $L$ to $\G_\even$, we show that the space of $\G_\even$-invariants $L^{\G_\even}$ is trivial. As a consequence we establish a conjecture first formulated by Gorelik, Grantcharov and Mazorchuk on the triviality of the supercharacter of irreducible $\G$-modules in the case when the modules are finite dimensional.
\end{abstract}

\subjclass[2010]{17B67}

\maketitle

\section{Introduction}

Let $\G=\G_\even\oplus\G_\odd$ be the queer Lie superalgebra. Let $\mf b$ be its standard Borel subalgebra and $\h=\h_{\bar 0}\oplus\h_{\bar 1}$ be its standard Cartan subalgebra (see Section \ref{sec:queer}). We denote by $\mc O$ the BGG category of $U(\mf b)$-locally finite, $\h_{\bar 0}$-semisimple, finitely generated $U(\G)$-modules. Unless otherwise stated, the morphisms in $\mc O$ are all (not necessarily even) morphisms.

For an object $M=M_{\bar 0}\oplus M_{\bar 1}$ in $\mc O$ and $\mu \in \mathfrak{h}_{\bar{0}}^*$, let $$M_{\mu}:=\{m\in M| h\cdot m = \mu(h)m, \text{ for } h\in \mathfrak{h}_{\bar{0}}\}$$ stand for its $\mu$-weight space. Thus, we have the weight space decomposition:
\begin{align*}
M=\bigoplus_{\mu\in\h^*_{\bar 0}}M_\mu=\bigoplus_{\mu\in\h^*_{\bar 0}}M_{\mu,\bar 0}\oplus \bigoplus_{\mu\in\h^*_{\bar 0}}M_{\mu,\bar 1}.
\end{align*}
The character and supercharacter of $M$ are
\begin{align*}
\text{ch}M=\sum_{\mu\in\h^*_\even}\dim M_\mu e^\mu,\quad \text{sch}M=\sum_{\mu\in\h^*_\even}\left(\dim M_{\mu,\bar 0}-\dim M_{\mu,\bar 1}\right) e^\mu,
\end{align*}
respectively, where $e$ is a formal indeterminate.

While the irreducible character problem has been studied extensively in the literature because of its obvious importance in the representation theory of the queer Lie superalgebra (see e.g.~\cite{Pe, PS2, Br2, Sv2, CC, CK, CKW, BD, SZ} for a sample), for supercharacters it is expected that the following should hold.

\begin{conjecture}\cite[Conjecture 2.4]{GG2}\label{conj:GGM} Let $L(\la)$ be the irreducible highest weight modules in $\mc O$ of highest weight $\la$. Suppose that $\la\not=0$. Then
\begin{align*}
\text{sch} L(\la)=0.
\end{align*}
\end{conjecture}
In \cite{GG2} Grantcharov and Gorelik also credit Mazorchuk for this conjecture, so we refer to this conjecture as the Grantcharov-Gorelik-Mazorchuk conjecture. An aim of this paper is to prove the following.

\begin{thm}\label{conj:GGM:finite} Let $L(\la)$ be the finite-dimensional irreducible module in $\mc O$ of highest weight $\la$.  Suppose that $\la\not=0$. Then
\begin{align*}
{\rm sch} L(\la)=0.
\end{align*}
\end{thm}

As every finite-dimensional $\G$-module is completely reducible over $\G_{\bar 0}$, we immediately have the following corollary of Theorem \ref{conj:GGM:finite}.

\begin{cor}\label{cor:aux1}
Let $L(\la)$ be the irreducible finite-dimensional irreducible module in $\mc O$ of highest weight $\la$.  Suppose that $\la\not=0$. Then as a $\G_{\bar 0}$-module we have $L(\la)_{\bar 0}\cong L(\la)_{\bar 1}$.
\end{cor}

One may ask if an analogous statement as Corollary \ref{cor:aux1} remains valid for infinite-dimensional $L(\la)\in\mc O$ as well \cite[2.2.1]{GG1}.

A problem related to Conjecture \ref{conj:GGM} turns out to be the following: Let $L(\la)$ be as above. We can restrict the action of $\G$ on $L(\la)$ to its even subalgebra $\G_\even\cong\gl(n)$ and consider $L(\la)$ as a $\gl(n)$-module. A main result in this article is the following.

\begin{thm}\label{thm:main}
Let $\la\not=0$ and let $L(\la)$ be the finite-dimensional irreducible $\G$-module of highest weight $\la$, which we regard as a $\G_{\bar 0}=\mf{gl}(n)$-module. Then the trivial $\G_\even$-module does not appear as a composition factor of $L(\la)$.
\end{thm}

In Section \ref{sec:proofs} we prove Theorem \ref{thm:main} using the combinatorics of Brundan's Kazhdan-Lusztig theory for finite-dimensional integer weight modules of the queer Lie superalgebras \cite{Br2}. Also, as we shall see in Proposition \ref{prop:aux1}, Theorem \ref{thm:main} indeed implies Theorem \ref{conj:GGM:finite}. However, we shall see in Remark \ref{rem:last} that an analogous result as Theorem \ref{thm:main} does not hold for infinite-dimensional $L(\la)$ in general, although it does hold for all but a finite number of irreducible modules (see Proposition \ref{lem:to:finite}). This in turn settles Conjecture \ref{conj:GGM} for all but a finite number of irreducible modules in the BGG category $\mc O$.

\subsection*{Notation} \label{SectionNotations} We use $\mathbb{N}$, $\mathbb{Z}$ and $\mathbb{Z}_+$ to denote the sets of natural numbers, integers, and non-negative integers, respectively. Here and below we let $m,n \in \mathbb{Z}_+$ and set \begin{align*} I(m|n):= \{ -m, -m+1, \ldots  ,-1\} \cup \{1, 2,\ldots ,n \}.\end{align*}
The category with the same objects as the BCC category $\mc O$, but with only even morphisms is denoted by $\mc O^\even$. The Grothendieck group of $\mc O$ will be denoted by $K(\mc O)$, etc.

Let $\mc C$ be a category of modules. For an object $M$ of finite length and an irreducible object $L$, we shall use $\left[M:L\right]$ to denote the number of composition factors of $M$ that are isomorphic to $L$ in $\mc C$.

\section{Preliminaries}\label{sec:prelim}

Besides setting up additional notations, in this section, we shall recall some basic results on the queer Lie superalgebra and its representation theory along with its connection with symmetric functions. To be more specific, we shall recall cohomological induction suitably adapted to the queer Lie superalgebra and Brundan's Kazhdan-Lusztig theory for the queer Lie superalgebra in Section \ref{sec:dual:zuck} and \ref{sec:Br:main}, respectively. Schur $P$-Laurent polynomials are recalled in Section \ref{sec:symm} together with their $\mf q(n)$-representation-theoretical interpretation.

\subsection{The queer Lie superalgebra}\label{sec:queer}
Let $\mathbb{C}^{m|n}$ be the complex superspace of super-dimension $(m|n)$ and let $\mathfrak{gl}(m|n)$ be the {\em general linear Lie superalgebra} of linear transformations on $\C^{m|n}$. Choosing homogeneous ordered bases $\{v_{-m},\ldots,v_{-1}\}$ and $\{v_1,\ldots,v_n\}$ for $\C^{m|0}$ and $\C^{0|n}$, respectively, we may realize $\gl(m|n)$ as $(m+n) \times (m+n)$ complex matrices of the form
\begin{align} \label{glrealization}
 \left( \begin{array}{cc} A & B\\
C & D\\
\end{array} \right),
\end{align}
where $A,B,C$ and $D$ are respectively $m\times m, m\times n, n\times m$, and $n\times n$ matrices. Let $E_{a,b}$ be the elementary matrix in $\mathfrak{gl}(m|n)$ with $(a,b)$-entry $1$ and all other entries 0, where $a,b\in I(m|n)$.

For $m=n$, the subspace
\begin{align} \label{qnrealization}
 \mf{g}:= \mathfrak{q}(n)=
\left\{ \left( \begin{array}{cc} A & B\\
B & A\\
\end{array} \right) \middle\vert\ A, B: \ \ n\times n \text{ matrices} \right\}
\end{align}
forms a subalgebra of $\mathfrak{gl}(n|n)$ called the {\em queer Lie superalgebra}.

The set $\{e_{ij}, \bar{e}_{ij}|1\leq i,j \leq n\}$ is a linear basis for $\mathfrak{g}$, where $e_{ij}= E_{-n-1+i,-n-1+j}+E_{i,j}$ and $\bar{e}_{ij}= E_{-n-1+i,j}+E_{i,-n-1+j}$. The even subalgebra $\mathfrak{g}_{\bar{0}}$ is spanned by $\{e_{ij}|1\leq i,j\leq n\}$, and hence is isomorphic to the general linear Lie algebra $\mathfrak{gl}(n)$. The odd subspace $\G_\odd$ is isomorphic to the adjoint module.


Let $\mathfrak{h} = \mathfrak{h}_{\bar{0}}\oplus \mathfrak{h}_{\bar{1}}$ be the standard Cartan subalgebra of $\mathfrak{g}$, with linear bases $\{e_{ii}| 1\leq  i \leq n\}$ and $\{\bar{e}_{ii}|1\leq i \leq n\}$ of $\mathfrak{h}_{\bar{0}}$ and $ \mathfrak{h}_{\bar{1}}$, respectively. Let $\{\delta_i| 1\leq i\leq n\}$ be the basis of $\mathfrak{h}_{\bar{0}}^{*}$ dual to $\{e_{ii}|1\leq i\leq n\}$. We define a symmetric bilinear form $( ,) $ on $\mathfrak{h}_{\bar{0}}^{*}$ by $( \delta_i,\delta_j)  = \delta_{ij}$, for $1\leq i,j\leq n$.

We denote by $\Phi$ the set of roots of $\mf{g}$ and by $\Phi^+$ the set of positive roots in its standard Borel subalgebra $\mf b=\mf b_\even\oplus\mf b_\odd$, which consists of matrices of the form \eqref{qnrealization} with $A$ and $B$ upper triangular. We have $\Phi^+ = \{\delta_i- \delta_j| 1\leq i<  j \leq n\}$. The root space $\G_\alpha$ has super-dimension $(1|1)$, for all $\alpha\in\Phi$. Set $$\rho=\hf\sum_{\alpha\in\Phi^+}\alpha.$$
The Weyl group of $\mathfrak{g}$ is the Weyl group of the reductive Lie algebra $\mathfrak{g}_{\bar{0}}$, which is the symmetric group $\mf S_n$ in $n$ letters. It acts naturally on $\mathfrak{h}_{\bar{0}}^*$ by permutation of the $\delta_i$s. For $\alpha=\delta_i-\delta_j\in\Phi^+$, $i<j$, we let $\ov{\alpha}:=\delta_i+\delta_j$.

Let $\la = \sum_{i=1}^{n} \la_i \delta_i$, $\la_i\in\C$, be an element in  $\mathfrak{h}_{\bar{0}}^*$, and consider the symmetric bilinear form on $\mathfrak{h}_{\bar{1}}^*$ defined by $\langle\cdot,\cdot\rangle_{\la} : = \la([\cdot,\cdot] )$. Let $\ell(\la)$ be the number of $i$'s with $\la_i \neq 0$. Let $\mathfrak{h}'_{\bar{1}}$ be a maximal isotropic subspace  of $\mathfrak{h}_{\bar{1}}$ associated to $\langle\cdot,\cdot\rangle_{\la} $. Put $\mathfrak{h}' =  \mathfrak{h}_{\bar{0}} \oplus \mathfrak{h}'_{\bar{1}}$. Let $\mathbb{C}v_{\la}$ be the one-dimensional $\mathfrak{h}'$-module of even parity with action defined by $h\cdot v_{\la} = \la(h)v_{\la}$ and $h' \cdot v_{\la} =0 $ for $h\in \mathfrak{h}_{\bar{0}}$, $h' \in \mathfrak{h}'_{\bar{1}}$. Then $$I_{\la} : = \text{Ind}_{\mathfrak{h}'}^{\mathfrak{h}}\mathbb{C}v_{\la}$$ is an irreducible $\mathfrak{h}$-module of dimension $2^{\lceil \ell(\la)/2 \rceil}$ (see, e.g., \cite[Section 1.5.4]{CW}), where here and below $\lceil \cdot \rceil$ is the ceiling function. We let $\Delta(\la): = \text{Ind}_{\mathfrak{b}}^{\mathfrak{g}}I_{\la}$ be the {\em Verma module}, where $I_{\la}$ is extended to a $\mathfrak{b}$-module in a trivial way, and define $L(\la)$ to be the unique irreducible quotient of $\Delta(\la)$. Note that $\Delta(\la)$ and $L(\la)$ are unique up to $\mf{g}$-isomorphisms in $\mc O$.

For $\la\in\h^*_\even$, let $\chi_\la$ denote its central character (see, e.g., \cite[Section 2.3.1]{CW}). Let $\la,\mu\in\h^*_\even$. We have $\chi_\la=\chi_\mu$, if and only if there exist $w\in\mf S_n$ and $\{\alpha_1,\ldots,\alpha_s\}\subseteq\Phi^+$ with $(\alpha_i,\alpha_j)=0$, for $i\not=j$, and complex numbers $k_i$, $1\le i\le s$, such that
\begin{align*}
\mu=w\left(\la-\sum_{i=1}^s k_i\alpha_i\right),
\end{align*}
where $(\la,\ov{\alpha}_i)=0$, for all $i=1,\ldots,s$ \cite{Sv} (see also \cite[Theorem 2.48]{CW}. We write $\la\succeq\mu$, if $\chi_\la=\chi_\mu$ and $\la-\mu\in\sum_{\alpha\in\Phi^+}\Z_+\alpha$.

Let $\Pi$ be the parity-reversing functor. Then $I_\la\cong\Pi I_\la$ in $\mc O^{\even}$ if and only if $\ell(\la)$ is odd, and hence $\Delta(\la)\cong\Pi\Delta(\la)$ in $\mc O^{\even}$ if and only if $\ell(\la)$ is odd.

Below we will often deal with $\G_\even$-modules as well.  In these occasions we shall use an upper-script $0$ to denote the corresponding $\gl(n)$-modules. For example, $\Delta^0(\la)$ and $L^0(\la)$ denote the $\G_\even$-Verma and irreducible module of highest weight $\la\in\h^*_\even$ with respect to the Borel subalgebra $\mf b_\even$, respectively.

\subsection{Dual Zuckerman functor}\label{sec:dual:zuck}

Let $\mf b\subseteq\mf p \subseteq \G$ be parabolic subalgebras with corresponding Levi subalgebras $\mf h\subseteq \mf l\subseteq \G$, respectively.
Let $\mc{HC}(\G,\mf{l}_\even)$ be the category of $\G$-modules that are direct sums of finite-dimensional simple $\mf{l}_\even$-modules, and let $\mc{HC}(\G,\mf{\G}_{\bar 0})$ be defined similarly.
Let $\mc L^{\G,\mf l}$ be the {\em dual Zuckerman functor} from $\mc{HC}(\G,\mf{l}_\even)$ to $\mc{HC}(\G,\G_\even)$ as in \cite[Section 4]{San}, which is a right exact functor. For $i\ge 0$, we denote by $\mc L^{\G,\mf l}_i$ its $i$th derived functor. For $M\in \mc{HC}(\G,\mf{l}_\even)$, we let
\begin{align*}
E^{\G,\mf l}(M):=\sum_{i\ge 0}(-1)^i\mc L^{\G,\mf l}_i(M)
\end{align*}
be the {\em Euler characteristic of $M$}, which we regard as a virtual $\G$-module.

By the same arguments as in \cite[Sections 4 and 5]{San}, one establishes the following.

\begin{prop}\label{prop:same:central}
Let $\la\in\h_\even^*$ be such that the irreducible $\mf l$-module $L(\mf{l},\la)$ with highest weight $\la$ is finite dimensional so that $M:={\rm Ind}_{\mf p}^{\G}L(\mf{l},\la)\in\mc{HC}(\G,\mf{l}_\even)$.
\begin{itemize}
\item[(1)] The $\G$-module $\mc L^{\G,\mf l}_0\left(M\right)$ is the maximal finite-dimensional quotient of $M$.

\item[(2)] The $\G$-module $\mc L^{\G,\mf l}_i\left(M\right)$ is finite dimensional for all $i\geq 0$, and $\mc L^{\G,\mf l}_i\left(M\right)=0$ for $i\gg 0$.

\item[(3)] The $\G$-modules $\mc L^{\G,\mf l}_i\left(M\right)$ have the same central character.

\item[(4)] The character of the Euler characteristic of $M$ is given by
\begin{align*}
{\rm ch} E^{\G,\mf l}(M)=D^{-1} \sum_{w\in \mf S_n}(-1)^{\ell(w)}w\left( \frac{{\rm ch}L(\mf{l},\la)}{\prod_{\alpha\in\Phi^+(\mf{l})}(1+e^{-\alpha})} \right),
\end{align*}
where $\Phi^+(\mf{l})$ denotes the set of positive roots of the queer Lie superalgebra $\mf l$ and
\begin{equation*}
D_{\bar 0}:=\prod_{\alpha\in\Phi^+}(e^{\alpha/2}-e^{-\alpha/2}),\quad
D_{\bar 1}:=\prod_{\alpha\in\Phi^+}(e^{\alpha/2}+e^{-\alpha/2}),\quad
D :=\frac{D_{\bar 0}}{D_{\bar 1}}.
\end{equation*}
\end{itemize}
\end{prop}

Define the set of {\em dominant} weights to be
$$\La^+:=\{\sum_{i=1}^n\la_i\delta_i\in\h^*_\even\vert\la_i-\la_{i+1}\in\Z_+\text{ for all }i\text{ and }\la_i=\la_{i+1}\text{ implies }\la_i=0\}.$$
For a weight $\la\in\h^*_\even$, it is known that $L(\la)$ is finite-dimensional if and only if $\la\in\La^+$ \cite[Theorem 4]{Pe}.

Let $\La_Z:=\{\sum_{i=1}^n\la_i\delta_i\in\h^*_\even\vert\la_i\in\Z\}\subseteq\h^*_\even$ be the set of integer weights and let $$\La^+_Z:=\La^+\cap\La_\Z$$
be the set of dominant integer weights.

For an element $\la\in\La^+_\Z$ we suppose that it is of the form: $$\la_1>\cdots>\la_l>\la_{l+1}=0=\cdots=0=\la_{k}>\la_{k+1}>\cdots>\la_n.$$
Let $\mf l(\la)=\h+\mf{q}(k-l)$ be the maximal Levi subalgebra of $\G$ such that $I_\la$ can be extended to an $\mf l(\la)$-module. Observe that $\mf{q}(k-l)$ acts trivially on $I_\la$. Set $\rho_{\mf l}=\hf\sum_{\alpha\in\Phi^+(\mf{l})}\alpha$. We compute the Euler characateristic of ${\rm Ind}_{\mf p}^\G L(\mf{l}(\la),\la)$ using Proposition \ref{prop:same:central}(4):
\begin{align}\label{Euler:P1}
\begin{split}
{\rm ch} E^{\G,\mf l}({\rm Ind}_{\mf p}^\G L(\mf{l}(\la),\la)) =& \prod_{\alpha\in\Phi^+}\frac{e^{\alpha/2}+e^{-\alpha/2}}{e^{\alpha/2}-e^{-\alpha/2}} \sum_{w\in \mf S_n}(-1)^{\ell(w)}w\left( \frac{{\rm ch}L(\mf{l},\la)}{\prod_{\alpha\in\Phi^+(\mf{l})}(1+e^{-\alpha})} \right)\\
=&\sum_{w\in \mf S_n} w\left(\prod_{\alpha\in\Phi^+}\frac{e^{\alpha/2}+e^{-\alpha/2}}{e^{\alpha/2}-e^{-\alpha/2}} \frac{{\rm ch}L(\mf{l},\la)}{\prod_{\alpha\in\Phi^+(\mf{l})}(1+e^{-\alpha})} \right)\\
=& {\rm dim}I_\la\sum_{w\in \mf S_n} w\left(\prod_{\alpha\in\Phi^+\setminus\Phi^+(\mf l)}\frac{e^{\alpha/2}+e^{-\alpha/2}}{e^{\alpha/2}-e^{-\alpha/2}}\prod_{\alpha\in\Phi^+(\mf{l})} \frac{e^{\la+\rho_{\mf l}}}{(e^{\alpha/2}-e^{-\alpha/2})} \right)\\
=2^{\lceil\ell(\la)/2\rceil}\sum_{w\in \mf S_n/\mf S_{k-l}} &w\left(\prod_{\alpha\in\Phi^+\setminus\Phi^+(\mf l)}\frac{e^{\alpha/2}+e^{-\alpha/2}}{e^{\alpha/2}-e^{-\alpha/2}}\cdot\sum_{\sigma\in\mf{S}_{k-l}}\sigma\left(\prod_{\alpha\in\Phi^+(\mf{l})} \frac{e^{\la+\rho_{\mf l}}}{(e^{\alpha/2}-e^{-\alpha/2})}\right) \right)\\
=& 2^{\lceil\ell(\la)/2\rceil}\sum_{w\in \mf S_n/\mf S_{k-l}} w\left(e^\la\prod_{\alpha\in\Phi^+\setminus\Phi^+(\mf l)}\frac{e^{\alpha/2}+e^{-\alpha/2}}{e^{\alpha/2}-e^{-\alpha/2}} \right),
\end{split}
\end{align}
where in the last identity we have used the Weyl character formula for the trivial $\gl(k-l)$-module.

Below we shall simplify the notation for $E^{\G,\mf l}({\rm Ind}_{\mf p}^\G L(\mf{l}(\la),\la))$ to $E(\la)$.

\subsection{Symmetric functions}\label{sec:symm}

For $\mu$ a generalized partition of length $n$, the Schur Laurent polynomial associated to $\mu$ is defined to be
\begin{align*}
s_\mu = \sum_{w\in\mf S_n}w\left(x_1^{\mu_1+n-1}x_2^{\mu_2+n-2}\cdots x_n^{\mu_n}\prod_{i<j}\frac{1}{x_i-x_j}\right).
\end{align*}
Recall that ${\rm ch}L^0(\mu)=s_\mu$, where we put $e^{\delta_i}:=x_i$, for all $1\le i\le n$.

Recall further that for $\la\in\La_\Z^+$ the Schur $P$-Laurent polynomial in $x_1,\ldots,x_n$ associated to $\la$ is defined by
\begin{align}\label{def:schurP}
P_\la=&\sum_{w\in\mf S_n/\mf S_\la}w\left(x_1^{\la_1}x_2^{\la_2}\cdots x_n^{\la_n}\prod_{i<j,\la_i>\la_j}\frac{x_i+x_j}{x_i-x_j}\right),
\end{align}
where $\mf S_n/\mf S_\la$ denotes the set of the minimal length left coset representatives of the stabilizer subgroup $\mf S_\la$ of $\la$ in $\mf S_n$.

Identifying $x_i$ with $e^{\delta_i}$ we conclude by comparing \eqref{Euler:P1} with \eqref{def:schurP} the following.

\begin{prop}\label{prop:aux2} (\cite[Section 1.3]{PS2},\cite[Theorem 4.11]{Br2}) Let $\la\in\La^+_\Z$. Then
$${\rm ch}E(\la)=2^{\lceil\ell(\la)/2\rceil}P_\la.$$
\end{prop}

Recall the Weyl denominator identity $\sum_{w\in\mf S_n}(-1)^{\ell(w)}e^{w(\rho)}=\prod_{\alpha\in\Phi^+}\left(e^{\alpha/2}-e^{-\alpha/2}\right)$, which in turn implies that
\begin{align*}
\sum_{w\in\mf S_n}(-1)^{\ell(w)}e^{2w(\rho)}=\prod_{\alpha\in\Phi^+}\left(e^{\alpha}-e^{-\alpha}\right).
\end{align*}
Hence we get by the Weyl character formula the well-known identity:
\begin{align}\label{ch:lrho}
{\rm ch}L^0\left(\rho\right)=\frac{\sum_{w\in\mf S_n}(-1)^{\ell(w)}e^{w(2\rho)}}{\prod_{\alpha\in\Phi^+}e^{\alpha/2}-e^{-\alpha/2}}= \prod_{\alpha\in\Phi^+}\left(e^{\alpha/2}+e^{-\alpha/2}\right).
\end{align}

Now let $\la\in\La^+_\Z$ be such that all the $\la_i$s in $\la\in\La^+_\Z$ are distinct. In this case we have $\mf S_\la=\{e\}$ and furthermore $\la-\rho$ is a generalized partition and hence we have
\begin{align}\label{P:distinct}
\begin{split}
P_\la=&\prod_{i<j}(x_i+x_j)\sum_{w\in\mf S_n}w\left(x_1^{\la_1}x_2^{\la_2}\cdots x_n^{\la_n}\prod_{i<j}\frac{1}{x_i-x_j}\right)\\
= &s_{\rho+\sum_{i=1}^n\frac{n-1}{2}\ep_i}\cdot s_{\la-\rho-\sum_{i=1}^n\frac{n-1}{2}\ep_i}\\
= & {\rm ch}L^0(\rho)\cdot{\rm ch}L^0(\la-\rho) = {\rm ch}\left[L^0(\rho)\otimes L^0(\la-\rho)\right],
\end{split}
\end{align}
where in the second equality above we have made use of \eqref{ch:lrho}.

Since the Schur Laurent polynomials form a basis for the space of symmetric Laurent polynomials, for any $\la\in\La^+_\Z$, we can write
\begin{align}\label{P:in:s}
P_\la=\sum_{\mu}g_{\mu\la}s_\mu,\quad g_{\mu\la}\in\Q,
\end{align}
where the sum is over all generalized partitions $\mu$ of length $n$.
When $\la$ is a strict partition, $P_\la$ is a symmetric polynomial, and in this case in the sum on the right hand side of \eqref{P:in:s} the non-zero summands are necessarily all Schur polynomials. Furthermore, in this case it is known that $g_{\mu\la}\in\Z_+$ \cite[Theorem 9.3(b)]{St}. Indeed, it follows from \cite{Br2} that in general the coefficients $g_{\mu\la}$ in \eqref{P:in:s} lie in $\Z_+$ as well. We shall see this in Remark \ref{rem:coeffs:P} below.

\subsection{Brundan's irreducible character formula}\label{sec:Br:main}

Let $\la=\sum_{i=1}^n\la_i\ep_i\in\La^+_\Z$. Let $p$ be the maximal integer such that there exist $1\le i_1<i_2<\cdots i_p<j_p<\cdots j_1\le n$ with $\la_{i_s}+\la_{j_s}=0$, for all $1\le s\le p$. Let $\texttt{I}_0:=\{|\la_i|,1\le i\le n\}$. For each $1\le s\le p$ we define the set $\texttt{I}_s$ and the integer $k_s$ inductively as follows: If $\la_{i_s}>0$, then let $k_s$ be the smallest positive integer such that $k_s>\la_{i_s}$ and $k_s\not\in \texttt{I}_{s-1}$. Define $\texttt{I}_s:=\texttt{I}_{s-1}\cup\{k_s\}$. If $\la_{i_s}=0$, let $k_s$ and $k'_s$ be the two smallest positive integers such that $k_s,k'_s>\la_{i_s}$ and $k_s,k'_s\not\in\texttt{I}_{s-1}$. In the case when $n-\ell(\la)$ is even, we let $k_s<k'_s$, while in the case $n-\ell(\la)$ is odd, we let $k'_s<k_s$. Define $\texttt{I}_s:=\texttt{I}_{s-1}\cup\{k_s,k'_s\}$.

Let $\theta=(\theta_1,\ldots,\theta_p)\in\Z_2^p$. Define $\texttt{R}_\theta(\la)$ to be the unique $\mf S_n$-conjugate in $\La^+_\Z$ of
\begin{align*}
\la+\sum_{s=1}^p\theta_s k_s\left(\ep_{i_s}-\ep_{j_s}\right)\in\La_\Z.
\end{align*}

The following is the Main Theorem in \cite{Br2}.

\begin{thm}(Brundan)\label{thm:Brundan}
Let $\mu\in\La^+_\Z$. Then in $K(\mc O)$ we have
\begin{align*}
[E(\mu)]=\sum_{\la}d_{\mu\la} [L(\la)],
\end{align*}
where the summation is over $\la\in\La^+_\Z$ such that there exists a $\theta_\la\in\Z_2^p$ with $\texttt{R}_{\theta_\la}(\la)=\mu$ and $d_{\mu\la}=2^{\frac{\ell(\mu)-\ell(\la)}{2}}$.
\end{thm}

\begin{rem}\label{rem:coeffs:P} According to Theorem \ref{thm:Brundan} the $d_{\mu\la}$s are always non-negative. Now restricting the $\G$-action on $L(\la)$ to $\G_\even$, we can express the character of $L(\la)$ as a non-negative integral linear combination of Schur Laurent polynomials. Finally, using Proposition \ref{prop:aux2} we see that the coefficients $g_{\mu\la}$ in \eqref{P:in:s} have representation-theoretical interpretation and hence they are always non-negative integers.
\end{rem}

\section{Proofs of Theorems \ref{conj:GGM:finite} and \ref{thm:main}}\label{sec:proofs}
The goal of this section is to establish the Gorelik-Grantcharov-Mazorchuk conjecture in the case when the module is finite dimensional. We start by studying the supercharacters of Verma modules.

\begin{lem}\label{lem:sch:verma}
Let $\la\in\h^*_{\bar 0}$ and $\Delta(\la)$ be the Verma module of highest weight $\la$. Then, we have
\begin{align*}
{\rm sch}\Delta(\la)=\begin{cases}
 0,\text{ if }\la\not=0,\\
 1, \text{ if }\la=0.
\end{cases}
\end{align*}
\end{lem}

\begin{proof}
We have $\text{ch}\Delta(\la)=\frac{\prod_{\alpha\in\Phi^+}(1+e^{-\alpha})}{\prod_{\alpha\in\Phi^+}(1-e^{-\alpha})}\dim I_\la$, and thus
\begin{align*}
{\rm sch}\Delta(\la) = \frac{\prod_{\alpha\in\Phi^+}(1-e^{-\alpha})}{\prod_{\alpha\in\Phi^+}(1-e^{-\alpha})}{\rm sch} I_\la = {\rm sch} I_\la.
\end{align*}
Now, the lemma follows, because ${\rm dim}\left(I_\la\right)_{\bar 0}={\rm dim}\left(I_\la\right)_{\bar 1}$, for $\la\not=0$.
\end{proof}

\begin{prop}\label{prop:aux1}
For any $\la\in\h^*_{\bar 0}$ we have ${\rm sch}L(\la)\in\Z$. Furthermore, if $\la\not\succeq 0$, then ${\rm sch}L(\la)=0$.
\end{prop}

\begin{proof}
The sets $\{[\Delta(\la)]\vert \la\in\h^*_{\bar 0}\}$ and $\{[L(\la)]\vert \la\in\h^*_{\bar 0}\}$ are two sets of basis for $K(\mc O)$. Now, if $\ell(\la)$ is odd, then $\Pi(L(\la))\cong L(\la)$ and also $\Pi(\Delta(\la))\cong\Delta(\la)$ in $\mc O^\even$. On the other hand, if $\ell(\la)$ is even then $\Pi(L(\la))\not\cong L(\la)$ and also $\Pi(\Delta(\la))\not\cong\Delta(\la)$ in $\mc O^\even$. Thus, we have the following identity in $K(\mc O^\even)$:
\begin{align*}
[L(\la)]=\sum_{\mu\preceq\la}a_{\mu\la}[\Delta(\mu)]+\sum_{\mu\preceq\la;\ell(\mu)\text{ even}}b_{\mu\la}[\Pi\left(\Delta(\mu)\right)],
\end{align*}
where $a_{\mu\la},b_{\mu\la}\in\Z$. The sum above is an infinite sum in general; however, for fixed $\nu\in\h^*_{\bar 0}$ and fixed $\epsilon\in\Z_2$ such that $(L(\la)_{\epsilon})_\nu\not=0$ there are only finitely many $\mu$s such that $(\Delta(\mu)_{\epsilon})_\nu\not=0$ with $a_{\mu\la}\not=0$ and only finitely many $\mu$s such that $(\Pi \Delta(\mu)_{\epsilon})_\nu\not=0$ with $b_{\mu\la}\not=0$ and $\ell(\mu)$ even.
By Lemma \ref{lem:sch:verma}, we have
\begin{align*}
{\rm sch}L(\la)&=\sum_{0\not=\mu\preceq\la}a_{\mu\la}{\rm sch}\Delta(\mu)+\sum_{0\not=\mu\preceq\la;\ell(\la)\text{ even}}b_{\mu\la}{\rm sch}\Pi\Delta(\mu)+a_{0\la}{\rm sch}\Delta(0)+b_{0\la}{\rm sch}\Pi\Delta(0)\\
=&a_{0\la}{\rm sch}\Delta(0)+b_{0\la}{\rm sch}\Pi\Delta(0) = a_{0\la}-b_{0\la}\in\Z.
\end{align*}
Now if $\la\not\succeq 0$, then $a_{0\la}=b_{0\la}=0$, and hence ${\rm sch}L(\la)=0$.
\end{proof}

We have the following immediate corollary.

\begin{cor}
Let $\la\in\h^*_\even\setminus\La_\Z$. Then ${\rm sch}L(\la)=0$.
\end{cor}

Thus, Conjecture \ref{conj:GGM} is reduced to the case of $\la\in\La_\Z$.

\begin{prop}\label{prop:aux4} Let $\{L(\la):L^0(\mu)\}$ and $\{L(\la):\Pi L^0(\mu)\}$ denote
the number of composition factors in $L(\la)$ isomorphic to the irreducible $\gl(n)$-module $L^0(\mu)$ of even and odd parity, respectively. Suppose that $\la\not=0$. Then $\{L(\la):L^0(\mu)\}=\{L(\la):\Pi L^0(\mu)\}$, for $\mu\not=0$. Furthermore,
\begin{align*}
{\rm sch}L(\la)=\{L(\la):L^0(0)\}-\{L(\la):\Pi L^0(0)\}.
\end{align*}
In particular, if $[L(\la):L^0(0)]=0$ in the category of $\G_\even$-modules, then ${\rm sch}L(\la)=0$.
\end{prop}

\begin{proof} Let $\{L(\la):L^0(\mu)\}=d_{\mu\la}$ and $\{L(\la):\Pi(L^0(\mu))\}=e_{\mu\la}$. Then we have the following identity in the Grothendieck group of the category of $\G_\even$-modules:
\begin{align*}
[L(\la)]=\sum_{\mu}d_{\mu\la}[L^0(\mu)]+\sum_{\mu}e_{\mu\la}[\Pi L^0(\mu)].
\end{align*}
By Proposition \ref{prop:aux1} we have ${\rm sch}L(\la)=\sum_{\mu}(d_{\mu\la}-e_{\mu\la}){\rm ch}L^0(\mu)\in\Z$. Since $\{{\rm ch}L^0(\mu)\vert\mu\in\h^*_\even\}$ is linearly independent, it follows that $d_{\mu\la}=e_{\mu\la}$, for $\mu\not=0$.  Thus, ${\rm sch}L(\la)=d_{0\la}-e_{0\la}$.
\end{proof}

\begin{lem}\label{lem:to:finite}
Let $\la\in\h_\even^*$ such that $\la\not=\sum_{\alpha\in I}\alpha$, for any subset $I\subseteq\Phi^+$. Then in the category of $\G_\even$-modules we have
\begin{align*}
\left[L(\la):L^0(0)\right]=0.
\end{align*}
In particular, ${\rm sch}L(\la)=0$.
\end{lem}

\begin{proof}
It suffices to prove that in the category of $\G_\even$-modules we have $\left[\Delta(\la):L^0(0)\right]=0$, for $\la$ as in the assumption of the lemma. For $\mu\in\h_\even^*$ recall that $\Delta^0(\mu)$ denotes the $\G_\even$-Verma module of $\mf b_\even$-highest weight $\mu$.

We compute
\begin{align*}
{\rm ch}\Delta(\la)=&2^{[\ell(\la)/2]}{\rm ch}\Delta^0(\la)\prod_{\alpha\in\Phi^+}(1+e^{-\alpha})\\
=& 2^{[\ell(\la)/2]}\sum_{I\subseteq\Phi^+}{\rm ch}\Delta^0(\la-\sum_{\alpha\in I}\alpha).
\end{align*}
Now $0$ is a dominant $\G_\even$-weight. Hence, in order for $L^0(0)$ to appear as a composition factor in the $\G_\even$-Verma module $\Delta^0(\la-\sum_{\alpha\in I}\alpha)$ we must have $\la-\sum_{\alpha\in I}\alpha=0$, and thus $\la=\sum_{\alpha\in I}\alpha$. The lemma follows.
\end{proof}

From Lemma \ref{lem:to:finite} it follows that, for any fixed $n$, Conjecture \ref{conj:GGM} holds for all but (possibly) a finite number of irreducible modules. Indeed, it was stated in \cite{GG2} that Conjecture \ref{conj:GGM} has been checked for all but finitely many irreducible modules, however, the set of those irreducibles for which the conjecture may fail was not given in \cite{GG2}.

Recall $\rho=\hf\sum_{\alpha\in\Phi^+}\alpha$ and we have $2\rho=\sum_{i=1}^n (n-2i+1)\ep_i\in\La^+_\Z$. Note that all the coefficients of $2\rho$ are distinct.

\begin{lem}\label{lem:typ:0}
Let $\la\in\La^+_\Z$ be such that all $\la_i$s are distinct, and write $P_\la=\sum_{\mu}g_{\mu\la}s_\mu$. Then we have $g_{0\la}=\begin{cases}
0,\text{ if }\la\not=2\rho,\\
1,\text{ if }\la=2\rho.
\end{cases}$
\end{lem}

\begin{proof}
Since all the $\la_i$s are distinct, $\la-\rho$ is a $\gl(n)$-dominant weight, and we have by \eqref{P:distinct}
\begin{align*}
P_\la={\rm ch}\left[L^0(\rho)\otimes L^0(\la-\rho)\right].
\end{align*}
Now $-w_0\rho=\rho$, where $w_0$ is the longest element in $\mf S_n$, and hence $L^0(\rho)^*\cong L^0(\rho)$. Thus, in order for $g_{0\la}\not=0$, we need to have $\rho=\la-\rho$ by Schur's lemma, and hence $\la=2\rho$, in which case $g_{0,2\rho}=1$.
\end{proof}

\begin{prop}\label{prop:aux3}
Let $\nu\in\La^+_\Z$ and assume that $\nu\not\succeq 0$. Regarding $L(\nu)$ as a $\G_\even$-module we have $\left[L(\nu):L^0(0)\right]=0$.
\end{prop}

\begin{proof}
By Theorem \ref{thm:Brundan} we have
\begin{align}\label{aux1}
\left[E(\la)\right]=\sum_{\mu\preceq\la}2^{\frac{{\ell(\la)}-\ell(\mu)}{2}}\left[L(\mu)\right],
\end{align}
where the sum is over all $\mu$s such that there exists $\theta_\mu$ with $\texttt{R}_{\theta_\mu}(\mu)=\la$. Taking characters we get by Proposition \ref{prop:aux2}
\begin{align*}
2^{\lceil\ell(\la)/2\rceil}P_\la={\rm ch}E(\la)=\sum_{\mu} 2^{\frac{{\ell(\la)}-\ell(\mu)}{2}}{\rm ch}L(\mu).
\end{align*}

Now, let $\nu\in\La^+_\Z$. We can find a $\theta$ such that $\texttt{R}_\theta(\nu)=\la$ with all the $\la_i$s distinct. Then $[L(\nu)]$ appears on the right-hand side of \eqref{aux1} as a non-zero summand. Now, if $0\not\preceq \nu$, then $0\not\preceq\la$, and hence we have $\la\not=2\rho$. By Lemma \ref{lem:typ:0} and Remark \ref{rem:coeffs:P} the trivial $\gl(n)$-module $L^0(0)$ does not appear as a composition factor of $L(\mu)$ for any of the $\mu$s in \eqref{aux1}. Thus, in particular $L^0(0)$ does not appear in $L(\nu)$ and hence $\left[L(\nu):L^0(0)\right]=0$.
\end{proof}

We are now in a position to prove Theorem \ref{thm:main}, which then together with Proposition \ref{prop:aux4} implies Theorem \ref{conj:GGM:finite}.

\begin{proof}[Proof of Theorem \ref{thm:main}]
Let $\la\in\La^+_\Z$. By Proposition \ref{prop:aux3} it remains to prove the case $\la\succeq 0$. We shall continue to use the notation of Section \ref{sec:Br:main}. Suppose that $\la_{i_{t+1}}=\cdots=\la_{i_p}=0$ and $\la_{i_{t}}\not=0$, for $1\le t\le p$. Define $\theta=(\theta_q)_{q=1}^p\in\Z_2^p$ by
\begin{align*}
\theta_{q}:=\begin{cases}
0,\text{ if }q=1,\ldots,t,\\
1,\text{ if }q=t+1,\ldots,p.
\end{cases}
\end{align*}
Define $\texttt{R}_\theta(\la):=\Psi\in\La^+_\Z$. Then all the $\Psi_i$s are distinct.

We distinguish two cases.

\noindent{\sc Case 1.} In the first case $\Psi\not=2\rho$. In this case we have
the following identity in $K(\mc O)$ by Theorem \ref{thm:Brundan}:
\begin{align}\label{eq:aux1}
[E(\Psi)]=\sum_{\nu}2^{\lceil (n-\ell(\nu)-1)/2\rceil}\left[L(\nu)\right],
\end{align}
where the sum is over all such $\nu$s for which there exists a $\theta_\nu$ with $\texttt{R}_{\theta_\nu}(\nu)=\Psi$. Since $\Psi\not=2\rho$, we have by Proposition \ref{prop:aux2} and Lemma \ref{lem:typ:0}
\begin{align*}
\left[E(\Psi):L^0(0)\right]=0.
\end{align*}
By Remark \ref{rem:coeffs:P} it follows that $\left[L(\nu):L^0(0)\right]=0$, for all $\nu$ on the right hand side of \eqref{eq:aux1}, and in particular $\left[L(\la):L^0(0)\right]=0$.

\noindent{\sc Case 2.} In the second case $\Psi=2\rho$. We first observe that choosing $\theta\in\Z_2^{\lceil (n-1)/2\rceil}$ such that $\theta_{q}=1$ for $1\le q\le\lceil (n-1)/2 \rceil$, we have $\texttt{R}_\theta(0)=2\rho$.
Thus, we have the following identity in $K(\mc O)$ by Theorem \ref{thm:Brundan}:
\begin{align}\label{eq:aux2}
[E(2\rho)]=2^{\lceil (n-1)/2\rceil}[L(0)]+\sum_{\nu}2^{\lceil (n-\ell(\nu)-1)/2\rceil}\left[L(\nu)\right],
\end{align}
where the sum is over all such $0\not=\nu$ such that there exists $\theta_\nu$ with $\texttt{R}_{\theta_\nu}(\nu)=2\rho$. In particular, we have
\begin{align*}
\left[E(2\rho):L(0)\right]=2^{\lceil (n-1)/2\rceil}.
\end{align*}
Now, ${\rm ch}E(2\rho)=2^{\lceil (n-1)/2\rceil} P_{2\rho}$ and by Lemma \ref{lem:typ:0} $[E(2\rho):L^0(0)]=2^{\lceil (n-1)/2\rceil}$ in the Grothendieck group of $\G_\even$-modules. This implies that
$\left[L(\nu):L^0(0)\right]=0$, for all $L(\nu)$ appearing in the right hand side of \eqref{eq:aux2}. Thus, in particular $\left[L(\la):L^0(0)\right]=0$.
\end{proof}

\begin{rem}\label{rem:last}
An analogous statement as Theorem \ref{thm:main} does not hold for irreducible modules in $\mc O$. For example, let $n=3$ and consider $\la=2\ep_1-\ep_2-\ep_3$. We have
\begin{align*}
{\rm ch}\Delta(\la)=4\left({\rm ch}\Delta^0(\la-\sum_{\alpha\in I}\alpha)\right),
\end{align*}
where the sum is over all subsets $I$ of $\{\ep_1-\ep_2,\ep_2-\ep_3,\ep_1-\ep_3\}$. From this, an easy calculation shows that in the category of $\G_\even$-modules we have
 $[\Delta(\la):L^0(0)]=4$. Clearly, $\la$ is typical and $\la\not\succeq 0$. Thus, we have $[\Delta(\la):L(0)]=0$. This implies that one of the $\mf{q}(3)$-composition factors must have a $\gl(3)$-composition factor equal to $L^0(0)$. Indeed, it is clear that for $\mu=w\la$, $w\in\mf S_3$, and $\mu\not=\la$ we must have $[L(\mu):L^0(0)]=0$ by weight consideration. It therefore follows that $\left[L(\la):L^0(0)\right]=4$.
\end{rem}

\end{document}